\RequirePackage{xcolor}
\documentclass[journal,twoside,web]{ieeecolor} 
\usepackage{lcsys}
\usepackage{dsfont} 
\usepackage[T1]{fontenc}
\usepackage[utf8]{inputenc}
\usepackage{graphicx}
\usepackage{amsmath,mathtools,booktabs}
\usepackage{amssymb}
\usepackage{xcolor,url}
\usepackage{algorithm,algorithmic}
\usepackage{cite}

\usepackage[caption=false,labelfont={footnotesize,bf,rm},textfont={footnotesize,rm},format=hang]{subfig}

\newtheorem{assumption}{Assumption}
\newtheorem{proposition}{Proposition}
\newtheorem{theorem}{Theorem}

\newtheorem{corollary}{Corollary}

\newtheorem{remark}{Remark}
\graphicspath{{./image/}}

\def\N {{\mathbb N}}

\begin{document}

\title{\LARGE \bf Sakawa-Shindo algorithm for optimal control of time-delay systems, with applications to epidemiology}

\author{Rami Katz$^{a}$, Francesca Cal\`a Campana$^{b}$ and Giulia Giordano$^{b}$
\thanks{R.K. is supported by the Alon Fellowship from the Council of Higher Education of Israel. 
G.G. and F.C.C. are supported by the European Union through the ERC INSPIRE grant (project n. 101076926).}% <-this % stops a space
\thanks{$^{a}$ School of Electrical and Computer Engineering, Tel Aviv University, Israel. E-mail: {\tt\small ramkatsee@tauex.tau.ac.il}.}
\thanks{$^{b}$ Department of Industrial Engineering, University of Trento, Italy. E-mails: {\tt\small \{f.calacampana,giulia.giordano\}@unitn.it}.}
}

%; views and opinions expressed are however those of the authors only and do not necessarily reflect those of the EU, the European Research Executive Agency or the European Research Council; neither the EU nor the granting authority can be held responsible for them

\maketitle
\thispagestyle{empty} % Removes the page number in the first page

\begin{abstract}
We extend the Sakawa-Shindo algorithm to solve optimal control problems where the system dynamics involve an arbitrary number of discrete state delays. We prove that the algorithm guarantees termination in a finite number of steps, asymptotic first-order optimality of the generated control sequence and convergence of a subsequence to a control satisfying first-order optimality, and we apply it to the optimal design of non-pharmaceutical interventions and vaccination plans for epidemic models with delays associated with incubation period and vaccination.
\end{abstract}

\section{Introduction}\label{sec:intro}
\IEEEPARstart{S}{olving} optimal control problems (OCPs) for delay differential equation (DDE) systems poses significant computational challenges, although the Pontryagin Minimum Principle (PMP) admits an extension to DDEs \cite{kolmanovskiui1996control,boccia2017maximum}. Numerical approaches developed to this aim include both direct methods, which reformulate the OCP as a nonlinear program %and optimise simultaneously over discretised state and control trajectories
\cite{ritschel2024}, and indirect methods, which discretise the PMP necessary conditions for optimality \cite{gollmann2008}.

Here, we consider OCPs governed by DDEs with an arbitrary number of discrete state delays (Section~\ref{sec:model}) and we propose the ESSA algorithm (Section~\ref{sec:ESSA}) that extends to a time-delay setting the Sakawa–Shindo (S\&S) algorithm \cite{Sakawa1980}, a PMP-based sequential numerical method originally developed for delay-free OCPs. In Section~\ref{sec:theoreticalresults}, we show well-posedness of the problem and prove that the algorithm terminates in a finite number of steps, achieves asymptotic first-order optimality of the generated sequence of controls, and guarantees converge of a subsequence to a control that satisfies first-order optimality.
In Section~\ref{sec:Num_exp}, we demonstrate the effectiveness of the ESSA algorithm to solve epidemiological control problems governed by DDE models. 
Epidemic models \cite{SIR,Brauer,Giordano2020,Giordano2021,HernandezVargas22} are precious to support the analysis and the control of epidemics, and optimal control is powerful to design interventions aimed at curbing the contagion \cite{Lenhart,Sharomi,Buonomo2019,CalaCampana2024}. Accounting for delays in epidemic dynamics is crucial \cite{Rihan2026} also for the solution of optimal control problems \cite{BASHIER2017,dong2023timedelay}. We consider the optimal design of non-pharmaceutical interventions and vaccination plans in the presence of time delays associated with the incubation period and the build-up of vaccine-induced protection.

\textbf{Notation.} 
For $k \in \mathbb{N}$, $[k]$ denotes the set $\{1, \dots, k\}$. Also, $\N_0 = \mathbb{N}\cup \{0 \}$.
We denote by $\|\cdot \|$ the Euclidean norm, as well as the corresponding induced operator norm.
$\mathcal{C}([a,b])$ is the space of continuous functions $f \colon [a,b] \to \mathbb{R}^n$ defined on $[a,b]$.
$L^p ([a,b],\mathbb{R}^n)$, $p\geq 1$, is the  space of Lebesgue measurable and $p$-integrable functions on $[a,b]$. $L^p ([a,b],\mathbb{R}_{\geq 0}^n)$ consists of $p$-integrable functions with $f(x)\geq 0$ a.e. component-wise. We denote by $\mathds{1}_A$ the indicator function of a measurable set $A$ and by $\chi(A)$ its Lebesgue measure. We use $\psi(\cdot,\mathrm{X}_h,u)\vert_t$ as an abbreviation for $\psi(t,\mathrm{X}_{h}(t),u(t))$.

\section{Optimal Control with Delays} \label{sec:model}

Let the initial time $t_0\in \mathbb{R}$ and time horizon $T\in \mathbb{R}$ satisfy $0 \leq t_0<T$ and the time delays $h_i \in \mathbb{R}$, $i\in [k]$, satisfy $0=:h_0<h_1<\dots<h_k=:h$.
We consider the system 
\begin{equation}\label{eq:Dynamics}
\begin{array}{lll}
\hspace{-8mm} &\dot{x}(t) = f(t,x(t-h_0),x(t-h_1),\dots ,x(t-h_k),u(t) ),\\
\hspace{-8mm} &x(t_0+\theta) = \phi(\theta),\ \theta \in [-h,0],
\end{array}
\end{equation}
where $t\in [t_0,T]$, the initial condition is $\phi\in \mathcal{C}([-h,0])$, the state is $x(t)\in \mathbb{R}^n$ and the control is $u(t)\in U \subseteq \mathbb{R}^m$; $U$ is a compact and convex set, the control input $u$ is a measurable function on $[t_0,T]$, and we denote the set of control functions $u \colon [t_0,T] \to U$ by $\mathcal{U}$.
We employ the abbreviations
\begin{equation*}
\begin{array}{lll}
\hspace{-5mm} &\mathrm{X} = (x_0,x_1,\dots,x_k )\in \mathbb{R}^{(k+1)n}, \text{  for  } \{x_i \}_{i=0}^k\subseteq \mathbb{R}^n,\\
\hspace{-5mm} &\mathrm{X}_{h}(t) = (x(t-h_0),x(t-h_1),\dots, x(t-h_k) )\in \mathbb{R}^{(k+1)n},
\end{array}
\end{equation*}
so that function $f \colon \mathbb{R}^{p}\to \mathbb{R}^n$, where $p=(k+1)n+m+1$, has arguments $(t,\mathrm{X},u)$ and the right-hand side of the time-delay system (TDS) in \eqref{eq:Dynamics} is $f(t,\mathrm{X}_{h}(t),u(t) )$.

Given $\ell \colon \mathbb{R}^{p}\to \mathbb{R}$, let the cost functional $J \colon \mathcal{U}\to \mathbb{R}$ be
\begin{equation}\label{eq:PerformanceInd}
J(u) = \int_{t_0}^{T} \ell (t,\mathrm{X}_{h}(t),u(t) )\mathrm{d}t.
\end{equation}

The functions $f$ and $\ell$ satisfy the following assumptions.

\begin{assumption}\label{assum:fandell}
Consider $\psi \in \{f, \ell \}$. For all $t$, $\psi(t,\cdot,\cdot)$ is twice continuously differentiable in $(\mathrm{X},u)$. Moreover, for all $(\mathrm{X},u)$, $\psi (\cdot, \mathrm{X},u)$ is measurable in $t$. Also, there exists $\mu\in L^2([t_0,T],\mathbb{R}_{\geq 0} )$ such that, for all $v\in U$ and all $\mathrm{X}\in \mathbb{R}^{(k+1)n}$, 
\begin{equation}\label{eq:CarCond}
\begin{array}{lll}
&\hspace{-5mm}\|\partial_{y}^{\alpha_y} \partial_{z}^{\alpha_z} \psi(t,\mathrm{X},v) \|\leq \mu(t)( \| \mathrm{X}\|+1),
\end{array}
\end{equation}
where $\alpha_y,\alpha_z \in \{0,1\}$ and $y,z \in \{x_0,x_1,\dots,x_k,u\}$.
\end{assumption}
%where $D^{\alpha}$ is \textit{any} differential operator in the variables $(\mathrm{X},u )$ of order $|\alpha|\leq 2$. {\color{red}Reformulate more specifically? I.e., we have 2 derivatives on any combination of $\mathrm{X},u$?}
% As shown below, our assumptions on $f$ and $\ell$ are sufficient to obtain well-posedness of \eqref{eq:Dynamics}. 
\begin{remark}\label{rem:terminal}
We do not explicitly include a terminal cost component just for simplicity of presentation; however, if $J(u)$ in \eqref{eq:PerformanceInd} is augmented with a sufficiently regular (e.g. Lipschitz) terminal cost $\gamma(x(T))$, by employing
\begin{equation*}
\begin{array}{lll}
&\hspace{-5mm}\gamma(x(T))-\gamma(x(t_0)) = \int_{t_0}^T \frac{\partial \gamma}{\partial x}(x(t) )f(t,\mathrm{X}_{h}(t),u(t))\mathrm{d}t,
\end{array}
\end{equation*}
we can present $J(u)$ as in \eqref{eq:PerformanceInd} with $\ell(\cdot,\mathrm{X}_{h},u)\vert_t$ replaced by $\ell_1(\cdot,\mathrm{X}_{h},u)\vert_t = \ell(\cdot,\mathrm{X}_{h},u)\vert_t+\frac{\partial \gamma}{\partial x}(x(t) )f(\cdot,\mathrm{X}_{h},u)\vert_t$.
\end{remark}

We consider the optimal control problem of finding
%$u^*\in \mathcal{U}$ such that $u^* = \operatorname{argmin}_{u\in \mathcal{U}}J(u)$ subject to \eqref{eq:Dynamics},
\begin{equation}\label{eq:optcont}
\begin{array}{ll}
u^* =& \operatorname{argmin}_{u\in \mathcal{U}}J(u) \quad \textrm{subject to \eqref{eq:Dynamics}}
\end{array}
\end{equation}
for $J(u)$ in \eqref{eq:PerformanceInd} via sequential numerical methods, and we show that the Sakawa-Shindo (S\&S) algorithm 
\cite{Sakawa1980}, which identifies extremal points of $J(u)$, can be extended to our TDS setting.

\section{Extended Sakawa-Shindo Algorithm (ESSA)}\label{sec:ESSA}
Denoting by $\lambda \in \mathbb{R}^n$ the co-state of the OCP \eqref{eq:optcont}, we introduce the Hamiltonian
% \footnote{An inspection of 
% \cite[Chapter 6]{kolmanovskii2012applied} shows that the Hamiltonian presented therein contains a typo. The correct Hamiltonian is given in \eqref{eq:HamDef}}
\begin{equation}\label{eq:HamDef}
\begin{array}{lll}
&\hspace{-5mm}H(t,\mathrm{X},u,\lambda) = (\ell(t,\mathrm{X},u)+\lambda^{\top}f(t,\mathrm{X},u ) ) \mathds{1}_{(-\infty,T]}(t).
% ,\ t\in [t_0,T],\\
% &\hspace{-8mm} H(t,\mathrm{X},u,\lambda)\equiv 0, \quad t>T,\ \forall(\mathrm{X},u,\lambda ).
\end{array}
\end{equation}
\begin{proposition}\label{Prop:HuCaratheodory}
$H(t,\cdot,\cdot,\cdot )$ is twice continuously differentiable in $(\mathrm{X},u,\lambda )$, for all $t$. $H(\cdot, \mathrm{X},u,\lambda )$ is measurable in $t$, for all $(\mathrm{X},u,\lambda )$. For all $v\in U$, $X\in \mathbb{R}^{(k+1)n}$ and $\lambda\in \mathbb{R}^n$,
\begin{equation}\label{eq:DHBound}
\begin{array}{lll}
\|\partial_y^{\alpha_y}\partial_z^{\alpha_z}H(t,\mathrm{X},v,\lambda ) \|\leq \mu(t)(1+\|\mathrm{X} \| )(1+\|\lambda \| )
\end{array}
\end{equation}
where $\alpha_y,\alpha_z\in \{0,1 \}$ and $y,z\in \{x_0,\dots,x_k,u,\lambda \}$. In particular, for $i\in \{0 \} \cup [k] $ and fixed $(\mathrm{X},u,\lambda )$, $H_{x_i}(\cdot,\mathrm{X},u,\lambda ) = \ell_{x_i}(\cdot,\mathrm{X},u)+\lambda^{\top}f_{x_i}(\cdot,\mathrm{X},u )$ is measurable.
% For all $i=0,\dots,k$, $H_{x_i}(t,\mathrm{X},u,\lambda )$ is well-defined. Moreover, (i) for all $t$, $H_{x_i}(t,\cdot,\cdot,\cdot)$ is continuously differentiable in $(\mathrm{X},u,\lambda)$; (ii) for all $(\mathrm{X},u,\lambda )$, $H_{x_i}(\cdot,\mathrm{X},u,\lambda )$ is measurable in $t$; and (iii) for all $v\in U$,
% where $D^{\alpha}$ is any differential operator in the variables $(\mathrm{X},u,\lambda )$ of order $|\alpha |\leq 1$.
\end{proposition}
\begin{proof}
The first three properties follow directly from Assumption~\ref{assum:fandell} and from the linear dependence of $H$ on $\lambda$, as given in \eqref{eq:HamDef}. Finally, $H_{x_i}(\cdot,\mathrm{X},u,\lambda )$ is measurable because it is the pointwise limit of measurable functions (written in terms of finite differences of $f$ and $\ell$).
% The fact that $H_{x_i}(t,\mathrm{X},u,\lambda ),\ i=0,\dots,k$ is well-defined and for all fixed $t$ it is continuous in $(\mathrm{X},u,\lambda )$ follows from our assumptions on $f$ and $\lambda$, as well as from the linear dependence of $H$ on $\lambda$.  Finally, \eqref{eq:DHBound} follows from \eqref{eq:CarCond}.
\end{proof}

%We make the following assumption regarding the convexity of the Hamiltonian in $u$.
\begin{assumption}\label{Assump:Huu}
There exists $R\succeq 0$ such that, for all $(t,\mathrm{X},u,\lambda )\in [t_0,T]\times \mathbb{R}^{p+n-1}$, it is 
$H_{uu}(t,\mathrm{X},u,\lambda )\succeq R$.
\end{assumption}

The Pontryagin Minimum Principle applies to our problem; see e.g. \cite[Theorem 3.5.2]{kolmanovskiui1996control} and the more general case in \cite{boccia2017maximum}.
\begin{theorem}\label{Thm:KMPMP}
If the OCP \eqref{eq:optcont} admits an optimal controller-trajectory pair $(u^*,x^*)$ with co-state $\lambda^* \colon [t,T]\to \mathbb{R}^n$, then, almost everywhere in $[t_0,T]$,
\begin{equation}\label{eq:PMP}
\begin{array}{lll}
&\hspace{-8mm}H(\cdot,\mathrm{X}^*_{h},u^*,\lambda^*)\vert_t = \min_{v\in U}H(t,\mathrm{X}^*_{h}(t),v,\lambda^*(t)),
\end{array}
\end{equation}
\begin{equation}\label{eq:Costate}
\begin{cases}
\dot{\lambda^*}(t) = -\sum_{j=0}^kH_{x_j}(\cdot,\mathrm{X}^*_{h},u^*,\lambda^* )\vert_{t+h_j},\\
\lambda^*(t)\equiv 0,\quad t \geq T. %s\geq T.
\end{cases}
\end{equation}
\end{theorem}

For a matrix $C\succ 0$, introduce the augmented Hamiltonian
    \begin{equation}\label{eq:AugHam}
    \begin{array}{lll}
    K(t,\mathrm{X},u,\lambda;v,C ) = H(t,\mathrm{X},u,\lambda )+(u-v )^{\top}C(u-v),
    \end{array}
    \end{equation}
which is strictly convex in $u$ by Assumption~\ref{Assump:Huu}.

Inspired by \cite{Sakawa1980}, we propose an extended Sakawa-Shindo algorithm (ESSA) to address the OCP \eqref{eq:optcont}.

\textbf{Step 0:} Select a nominal control $u^0\in \mathcal{U}$ and let $x^0$ be the corresponding solution of \eqref{eq:Dynamics}. Choose a diagonal matrix $C^1\succ 0$ and a tolerance $\eta_{\text{tol}}>0$. Set $i=1$.

\textbf{Step 1:} Given $u^{i-1}$ and $x^{i-1}$,  compute $\lambda^{i-1}$ by solving the TDS \eqref{eq:Costate} with $\mathrm{X}_{h}^*$, $u^*$ and $\lambda^*$ replaced with $\mathrm{X}_{h}^{i-1}$, $u^{i-1}$ and $\lambda^{i-1}$, respectively.

\textbf{Step 2:} Compute $x^i(t)$ and $u^i(t)$ by solving simultaneously
\begin{equation}\label{eq:Step2Simult}
\begin{array}{lll}
&\hspace{-8mm}\dot{x}^i(t) = f(\cdot,\mathrm{X}^i_{h},u^{i})\vert_t;\  x^i(t_0+\theta)=\phi(\theta),\ \theta\in [-h,0], \\
&\hspace{-8mm}K(\cdot,\mathrm{X}^i_{h},u^i,\lambda^{i-1};u^{i-1},C^i )\vert_t \\
&\hspace{2mm}= \min_{v\in U}K(t,\mathrm{X}^i_{h}(t),v,\lambda^{i-1}(t);u^{i-1}(t),C^i ).
\end{array}
\end{equation}

\textbf{Step 3:} If $\|u^i-u^{i-1} \|^2_{L^2([t_0,T])} \leq \eta_{\text{tol}}$, stop. Otherwise, compute $J(u^i)$. If $J(u^i)\geq J(u^{i-1})$, then increase all diagonal entries of $C^i$, and go back to Step 2; else, set $i\leftarrow i+1$ and $C^{i+1}=C^i$, and go back to Step 1. 

Corollary~\ref{Cor:Termin} in Section~\ref{sec:theoreticalresults} and the subsequent discussion show that the algorithm terminates in a finite number of steps. 

\textbf{Implementation:} Steps 1 and 2 are implemented by discretizing the time domain $[t_0,T]$ via a dense grid $\mathcal{G}$, given by $t_0=\rho_0<\rho_1<\dots <\rho_N=T_0$. Then, $u^i$ is obtained via minimisation on the grid, where $u^i\vert_{[\tau_j,\tau_{j+1})}\equiv u(\tau_j)$ for $j\in \{0 \}\cup[N-1]$, whereas the TDSs in \eqref{eq:Costate} and \eqref{eq:Step2Simult} are solved by a finite-difference scheme, via the step method. Thus, the computed controller $u^i$ is piecewise constant. See \cite{Sakawa1980}, and Propositions~\ref{Prop:WPx} and~\ref{Prop:LambdBound} in Section~\ref{sec:theoreticalresults}, for further details. 

\section{Theoretical Guarantees for ESSA}\label{sec:theoreticalresults}
%The aim of this section is to show that the guarantees of the S\&S algorithm in \cite{Sakawa1980} are guaranteed to hold for the ESSA.

The following proposition guarantees well-posedness of \eqref{eq:Dynamics} for any $u
\in \mathcal{U}$, and uniform boundedness of the solutions.
\begin{proposition}\label{Prop:WPx}
For fixed $u\in \mathcal{U}$ and $\phi\in \mathcal{C}([-h,0])$, there exist an absolutely continuous function $x$ that solves \eqref{eq:Dynamics} and some $M_{1}>0$ such that $\|x \|_{L^{\infty}([t_0,T]) }\leq M_{1}$ uniformly in $\mathcal{U}$.
\end{proposition}
\begin{proof}
The intersection of the sets $\{t_0+kh_i \}_{k\in \mathbb{N}_0,i\in[k]}$ and $[t_0,T]$ induces a partition of the interval of the form $t_0=\tau_1<\dots<\tau_N=T$, such that, for all $j\in [N-1]$ and $t\in [\tau_j,\tau_{j+1}]$, we have $t-h_i \leq \tau_j$. We apply the step method over the sub-intervals $\{[\tau_j,\tau_{j+1}] \}_{j=1}^{N-1}$.

Consider $[t_0,\tau_1]$. Since $h_1<\dots<h_k$, we have $[t_0,\tau_1]=[t_0,t_0+h_1]$ and, on the interval, the TDS \eqref{eq:Dynamics} has the form
\begin{equation}\label{eq:FirstInterv}
\hspace{-2mm}  \dot{x}(t) = f(\cdot,x,\phi(\cdot-t_0-h_1),\dots ,\phi(\cdot-t_0-h_k),u )\vert_t. 
\end{equation}
Denote the right-hand side by $g_1(t,x(t))$. 
By Assumption~\ref{assum:fandell} on $f$, $g_1(t,x)$ is continuous in $x$ for any fixed $t$. Moreover, the functions $\phi (\cdot-t_0-h_i )$, $i\in [k]$, are continuous, whereas $u$ is measurable and bounded, since $U$ is compact; thus, $\phi (\cdot-t_0-h_i )$, $i\in [k]$, and $u$ are uniform limits of simple measurable functions, whence, by Assumption~\ref{assum:fandell} on $f$, for a fixed $x$,
$g_1(t,x)$ is the pointwise limit of measurable functions, whence measurable. Also, by \eqref{eq:CarCond} and continuity of $\phi$, for any compact $K\subset \mathbb{R}^n$ there exists a constant $\mathcal{D}_K$ such that $\| g_1(t,x)\|\leq \mathcal{D}_K\mu(t)$, for all $(t,x)\in [0,\infty)\times K$. Hence, by the Carathéodory existence theorem \cite{hale2009ordinary}, \eqref{eq:FirstInterv} with initial condition $x(t_0)=\phi(0)$ has a local Carathéodory solution on $[t_0,t_0+\epsilon)$ with $\epsilon>0$. In addition, since \eqref{eq:CarCond} and the mean value theorem imply the existence of some $\mathcal{D}^1_{K}>0$ such that $\|g_1(t,x)-g_1(t, \hat x ) \|\leq \mathcal{D}_K^1 \mu(t)\|x- \hat x \|$ for all $(t,x),(t,\hat x)\in [0,\infty)\times K$, the local solution is unique. We now show that on $[t_0,t_0+h_1]$, the solution remains bounded in a compact; thus, it can be extended up to $t=t_0+h_1$. For $t\in [t_0,t_0+\epsilon)$ , the combination of \eqref{eq:CarCond} with \eqref{eq:FirstInterv} yields
\begin{equation*}
\begin{array}{lll}
&\|x(t) \| \leq \|\phi(0) \| + \int_{t_0}^{t}\|g(s,x(s)) \|\mathrm{d}s\\
&\qquad \leq \alpha+\int_{t_0}^t\mu(s)\|x(s) \|\mathrm{d}s,
\end{array}
\end{equation*}
where $\alpha = k (1+\|\phi \|_{L^\infty([-h,0])} ) (1+\|\mu \|_{L^1([t_0,T])})$.
By the Grönwall-Bellman inequality, $\|x(t) \|\leq \alpha \mathrm{e}^{\| \mu\|_{L^1([t_0,T])}}$. Next, consider \eqref{eq:Dynamics} on $[\tau_1,\tau_2]$ with the initial condition $x(\tau_1)=x(t_0+h_1)$ obtained from the previous step. Noting that $t\in [\tau_1,\tau_2]$ implies that $t-h_i \leq \tau_1$, $i\in [k]$, we can repeat the same arguments to obtain the existence of an absolutely continuous solution on $[\tau_1,\tau_2]$. The patching of the solutions on $[\tau_0,\tau_1]$ and $[\tau_1,\tau_2]$ is again an absolutely continuous solution on $[\tau_0,\tau_2]$, in view of continuity at $t=\tau_1$. Continuing inductively, we obtain the claim of the proposition.
\end{proof}
\begin{corollary}\label{cor:CostBounded}
Functional $J(u)$, $u\in \mathcal{U}$, in \eqref{eq:PerformanceInd} is bounded. 
\end{corollary}
\begin{proof}
Consider $u\in \mathcal{U}$ and the corresponding solution $x$ of \eqref{eq:Dynamics}.
By Assumption~\ref{assum:fandell} on $\ell$, we have that for $t\in [t_0,T]$, $|\ell(\cdot,\mathrm{X}_{h},u )\vert_t |\leq \mu(t)(1+ \sqrt{k+1} M_1)$. Integrating the latter, we have $|J(u) |\leq \|\mu \|_{L^1([t_0,T] )}(1+\sqrt{k+1} M_1)$.
\end{proof}

A similar well-posedness result to Proposition~\ref{Prop:WPx} holds for the TDS satisfied by $\lambda^{i-1}$ in \eqref{eq:Costate} in \textbf{Step 1} of the algorithm.
\begin{proposition}\label{Prop:LambdBound}
Given $u\in \mathcal{U}$ and an absolutely continuous $x \colon [t_0,T]\to \mathbb{R}^n$ with $\|x \|_{L^{\infty}([t_0,T])}\leq M_{1}$, system \eqref{eq:Costate} with $\mathrm{X}_{h}^*$, $u^*$ and $\lambda^*$ replaced by $\mathrm{X}_{h}$, $u$ and $\lambda$, respectively, has an absolutely continuous solution $\lambda$. Also, $\|\lambda\|_{L^{\infty}([t_0,T] )}\leq M_2$, uniformly over $u \in \mathcal{U}$ and $x$ satisfying our assumptions, for some $M_2>0$. 
\end{proposition}
\begin{proof}
The proof is similar to that of Proposition~\ref{Prop:WPx}. Let $t_0=s_1<\dots <s_R=T$ be such that, for all $j\in [R-1]$ and $t\in [s_j,s_{j+1}]$, it is $t+h_i \geq s_{j+1}$, $i\in [k]$. We employ the step method on $\{[s_j,s_{j+1}] \}_{j=1}^{R-1}$ backwards. In the interval $[s_{R-1},s_R] = [T-h_1,T]$, \eqref{eq:Costate} has the form $\dot{\lambda}(t) = H_{x_0}(t\cdot,\mathrm{X}_{h},u,\lambda )\vert_{t}$, with $\lambda(T)=0$:
% \begin{equation*}
% \dot{\lambda}(t) = H_{x_0}(t,\mathrm{X}_{h}(t),u(t),\lambda(t) ), \ \lambda(T)=0,
% \end{equation*}
$t+h_i \geq T$, $i\in [k]$, implies that $H_{x_i}(\cdot,\mathrm{X}_{h},u,\lambda )\vert_{t+h_i}\equiv 0,\ i\in [k]$. The local existence of a Carathéodory solution, uniqueness and extensibility to $[s_{R-1},s_R]$ follow as in Proposition~\ref{Prop:WPx}, by employing Proposition~\ref{Prop:HuCaratheodory} and the Grönwall-Bellman inequality, together with \eqref{eq:DHBound} and $\|x \|_{L^{\infty}([t_0,T])}\leq M_{1}$. Then, repeating the arguments on the previous sub-intervals yields the result.
\end{proof}

Propositions~\ref{Prop:WPx} and~\ref{Prop:LambdBound} yield well-posedness of the implementation of the ESSA.
Differently from \cite{Sakawa1980}, our more general assumptions on $f$ and $\ell$ yield a bound on the $L^2$ norm of the discrepancy $\Delta_u^i :=u^i-u^{i-1}$ along the ESSA iterations.
\begin{proposition}\label{Prop:Regulariz}
There exists $M_3>0$ such that, for all $i\in \mathbb{N}$, $\varepsilon_{min}^i\|\Delta_u^i \|_{L^2([t_0,T])}^2\leq M_{3}$, where $\varepsilon_{min}^i>0$ is the smallest eigenvalue of the diagonal matrix $C^i \succ 0$. 
\end{proposition}
\begin{proof}
In view of \eqref{eq:AugHam} and the minimisation in \textbf{Step 2} of the algorithm, for a fixed $t$, we have $H(\cdot,\mathrm{X}_{h}^i,u^i,\lambda^{i-1})\vert_t + (\Delta_u^i)^\top C^i (\Delta_u^i)\vert_t \leq H(\cdot, \mathrm{X}_{h}^i,u^{i-1},\lambda^{i-1})$ and hence, by the mean value theorem,
\begin{equation}\label{eq:DiscUppBd}
\begin{array}{lll}
&\hspace{-5mm}\varepsilon_{min}^i\|\Delta_u^i \|^2\vert_t\leq (\int_0^1\varphi^i(s)\mathrm{d}s) \Delta_u^i \vert_t 
\end{array}
\end{equation}
where $\varphi^i(s) = H_u(\cdot,\mathrm{X}_{h}^i,su^{i-1}+(1-s)u^{i},\lambda^{i-1} )\vert_t$. We have  $\|\varphi^i(s) \|\leq \mu(t) (1+ \sqrt{K+1} M_1 )(1+M_2 )=:M_{4}\mu(t)$ by convexity of $U$, Propositions~\ref{Prop:WPx} and~\ref{Prop:LambdBound}, and arguments similar to the proof of Proposition~\ref{Prop:HuCaratheodory}. Therefore,
\begin{equation}\label{eq:DiscIneq}
\begin{array}{lll}
&\hspace{-5mm}\varepsilon_{min}^i\|\Delta_u^i \|^2\vert_t \leq M_{4}\mu(t)\|\Delta_u^i \|\vert_t\\
&\hspace{10mm}\leq \frac{M_{4}^2}{2\varepsilon^0
_{min}}\mu(t)^2+ \frac{\varepsilon^i_{min}}{2}\|\Delta_u^i \|^2\vert_t,
\end{array}
\end{equation}
as 
% $2ab\leq a^2+b^2$ and 
$\{\varepsilon_{min}^i\}_{i=0}^{\infty}\subseteq \mathbb{R}_{>0}$ is  non-decreasing. Integrating \eqref{eq:DiscIneq} over $[t_0,T]$ yields the result with $M_3=\frac{M_{4}^2}{\varepsilon^0_{min}}\|\mu \|^2_{L^2([t_0,T])}$.
\end{proof}
Proposition~\ref{Prop:Regulariz} shows that, by choosing $C^0$ with $\varepsilon^0_{min}\gg 1$, the ESSA can be regularised numerically so that the $L^2([t_0,T])$ errors between consecutive controllers are not large. Moreover, recalling the proposed \textbf{Implementation}, we have a bound on the $L^{\infty}([t_0,T])$ error between consecutive controllers, depending on the minimal separation of nodes in the grid $\mathcal{G}$.
\begin{corollary}\label{Cor:Disc}
For the proposed numerical \textbf{Implementation} with an arbitrary grid $\mathcal{G}$, denote $\mathfrak{d}_{\mathcal{G}} = \min_{j\in [N]}(\rho_j-\rho_{j-1} )$. If $\mathfrak{d}_{\mathcal{G}}\geq \mathfrak{d}_{*}>0$, then
$\varepsilon^i_{min}\|\Delta_u^i \|_{L^{\infty}([t_0,T] )}^2\leq \mathfrak{d}_*^{-1}M_3$.
\end{corollary}
\begin{proof}
Since $u^i,u^{i-1}$ are piecewise constant, for some $j\in [N]$ it is $\|\Delta_u^i \|_{L^{\infty}([t_0,T] )}=|u^i(\rho_{j-1})-u^{i-1}(\rho_{j-1}) |$. Thus, for $\mathcal{I}_j=[\rho_{j-1},\rho_j]$,
since $\rho_j-\rho_{j-1} \geq \mathfrak{d}_\mathcal{G} \geq \mathfrak{d}_*$,
we have
$\varepsilon_{min}^i\|\Delta_u^i \|_{L^{\infty}([t_0,T] )}^2 \leq \frac{\varepsilon_{min}^i}{\mathfrak{d}_{\mathcal{G}}}\|\Delta_u^i \|_{L^{2}(\mathcal{I}_j )}^2 \leq \mathfrak{d}_*^{-1}M_3$.
\end{proof}
%Corollary~\ref{Cor:Disc} yields a bound on the $L^{\infty}([t_0,T])$ error between consecutive controllers, depending on the minimal separation of nodes in the grid $\mathcal{G}$.

% \begin{remark}
% Note that if we only assume $\mu\in L^1[t_0,T]$, Propositions~\ref{Prop:WPx} and~\ref{Prop:LambdBound} remain true, whereas in Proposition~\ref{Prop:Regulariz}, we only need to replace $L^2[t_0,T]$ with $L^1[t_0,T]$. Indeed, one simply needs to consider the first inequality in \eqref{eq:DiscIneq} for $t$ such that $u^i(t)\neq u^{i-1}(t)$, divide by $\|u^{i}(t)-u^{i-1}(t) \|$, and integrate obtained inequality on $[t_0,T]$.
% \end{remark}
We now show that the cost decreases along the iterations.
\begin{proposition}\label{Prop:JDecrease}
Given Assumption~\ref{Assump:Huu}, denote by $r_{min}>0$ the smallest eigenvalue of $R$.
There exists $M_5>0$ such that, for all $i\in \mathbb{N}_0$ and $\xi_i = 2\varepsilon_{min}^i+\frac{1}{2}r_{min}-M_5$, it holds
\begin{equation}\label{eq:JDecrease}
J(u^i )-J(u^{i-1})\leq -\xi_i\|\Delta_u^i\|_{L^2([t_0,T])}^2.
\end{equation}
In particular, if $\varepsilon_{min}^0$ is large enough, $\{J(u^i) \}_{i=0}^{\infty}$ is monotonically decreasing and, thus, convergent.
\end{proposition}
\begin{proof}
Denote
$\Delta_{x,s}^i  = x^i(\cdot-s)-x^{i-1}(\cdot-s)$.
% \begin{equation*}
% \begin{array}{lll}
% &\hspace{-5mm}\Delta_{x,s}^i  = x^i(\cdot-s)-x^{i-1}(\cdot-s),
% % \\ 
% % &\Delta_{\mathrm{X}_{h}}(t) = \mathrm{X}^i_{h}(t)-\mathrm{X}_{h}^{i-1}(t),
% \text{ and } \Delta_u^i=u^i-u^{i-1}.
% \end{array}
% \end{equation*}
By \eqref{eq:Dynamics}, \eqref{eq:PerformanceInd} and \eqref{eq:HamDef}, we have 
\begin{equation*}
\begin{array}{lll}
&\hspace{-5mm}J(u^i)-J(u^{i-1}) = -\int_{t_0}^T(\lambda^{i-1}(t))^{\top}\dot{\Delta}^i_{x,h_0}(t)\mathrm{d}t\\
&\hspace{-4mm}+\int_{t_0}^T(H(\cdot,\mathrm{X}_{h}^i,u^i,\lambda^{i-1} )-H(\cdot,\mathrm{X}_{h}^i,u^{i-1},\lambda^{i-1} ))\vert_t \mathrm{d}t\\
&\hspace{-4mm}+\int_{t_0}^T(H(\cdot,\mathrm{X}_{h}^i,u^{i-1},\lambda^{i-1} )-H(\cdot,\mathrm{X}_{h}^{i-1},u^{i-1},\lambda^{i-1} ) )\vert_t\mathrm{d}t\\
&\hspace{-3mm}=:\mathcal{J}_1+\mathcal{J}_2+\mathcal{J}_3,
\end{array}
\end{equation*}
We treat the three integrals separately.

For $\mathcal{J}_2$, using Proposition~\ref{Prop:HuCaratheodory} and \eqref{eq:AugHam}, we can expand the integrand in a Taylor polynomial of second order around $u^i$:
\begin{equation}\label{eq:DiffH}
\begin{array}{lll}
&\hspace{-6mm}(H(\cdot,\mathrm{X}_{h}^i,u^i,\lambda^{i-1} )-H(\cdot,\mathrm{X}_{h}^i,u^{i-1},\lambda^{i-1} ))\vert_t \\
% &\hspace{-5mm} =(H_u(\cdot, \mathrm{X}_{h}^i,u^{i},\lambda^{i-1})\Delta_u-\frac{1}{2}\Delta_u^{\top}H_{uu}(\cdot,\mathrm{X}_{h}^i,\hat{u},\lambda^{i-1} )\Delta_u)\vert_t \\
&\hspace{-6mm}=(K_u(\cdot,\mathrm{X}_{h}^i,u^i,\lambda^{i-1};u^{i-1},C^i )\Delta_u^{i}-2(\Delta_u^i)^{\top}C^i\Delta_u^i \\
& -\frac{1}{2}(\Delta_u^i)^{\top}H_{uu}(\mathrm{X}_{h}^i,\hat{u},\lambda^{i-1} )\Delta_u^i)\vert_t, 
\end{array}
\end{equation}
for some $\hat{u}\in U$, by convexity of $U$.
Due to the minimisation in \textbf{Step 2}, the first term on the right-hand side is non-positive a.e. on $[t_0,T]$. Thus, integrating \eqref{eq:DiffH} and employing Assumption~\ref{Assump:Huu}, we have that $\mathcal{J}_2 \leq -\frac{1}{2}(4\varepsilon_{min}^i+r_{min})\|\Delta^i_u \|_{L^2([t_0,T])}^2$.

For $\mathcal{J}_3$, using Proposition~\ref{Prop:HuCaratheodory}, we can expand the integrand in a Taylor polynomial of second order around $\mathrm{X}_{h}^{i-1}$:
\begin{equation}\label{eq:Tay2}
\begin{array}{lll}
&(H(\cdot,\mathrm{X}_{h}^i,u^{i-1},\lambda^{i-1} )-H(\cdot,\mathrm{X}_{h}^{i-1},u^{i-1},\lambda^{i-1} ) ) \vert_t\\
&\hspace{2mm}= ( \frac{1}{2}(D^i)^{\top}\operatorname{Hess}_{\mathrm{X},H}(\cdot,\hat{\mathrm{X}},u^{i-1},\lambda^{i-1})D^i\\
% &\hspace{14mm}+H_{x_0}(\mathrm{X}_{h}^{i-1},u^{i},\lambda^{i-1} )\Delta_{x,0}\\
&\hspace{4mm}+ \sum_{j=0}^kH_{x_j}(\cdot,\mathrm{X}_{h}^{i-1},u^{i-1},\lambda^{i-1} )\Delta^i_{x,h_j} )\vert_t,
\end{array}
\end{equation}
where $D^i=\operatorname{col}\{\Delta^i_{x,h_0},\Delta^i_{x,h_1},\dots,\Delta^i_{x,h_k} \}$, $\operatorname{Hess}_{\mathrm{X},H}$ is the Hessian of $H$ with respect to $\mathrm{X}$ (see Proposition~\ref{Prop:HuCaratheodory}) and $\hat{\mathrm{X}}$ lies on the line between $\mathrm{X}_{h}^i(t)$ and $\mathrm{X}_{h}^{i-1}(t)$. By Propositions~\ref{Prop:HuCaratheodory}-\ref{Prop:LambdBound}, for some constants $M_6,M_7>0$,
\begin{equation*}
\begin{array}{lll}
&\hspace{-5mm}\int_{t_0}^T ((D^i)^{\top}\operatorname{Hess}_{\mathrm{X},H}(\cdot,\hat{\mathrm{X}},u^{i-1},\lambda^{i-1})D^i)\vert_t \mathrm{d}t\\
&\leq M_6\int_{t_0}^T\mu(t)(\sum_{j=0}^k \|\Delta_{x,h_j}^i\|^2)\vert_t\mathrm{d}t\\
&\leq M_7 \|\mu \|_{L^1([t_0,T])}\| \Delta_{x,h_0}^i\|^2_{L^{\infty}([t_0,T])},
% \int_{t_0}^T\mathcal{M}(t)\|\Delta_{x,0} \|^2 \mathrm{d}t,
\end{array}
\end{equation*}
where the last inequality follows from the fact that the initial condition $\phi$ is fixed, whence, for each $j\in \{0\} \cup [k]$,
\begin{equation*}
\begin{array}{lll}
&\int_{t_0}^T\mu(t)\|\Delta^i_{x,h_j}(t) \|^2\mathrm{d}t =\int_{h_j}^T\mu(t)\|\Delta^i_{x,h_0}(t-h_j) \|^2\mathrm{d}t\\
&\hspace{30mm}\leq \|\mu \|_{L^1([t_0,T])}\|\Delta^i_{x,h_0} \|^2_{L^{\infty}([t_0,T])}.
\end{array}
\end{equation*}
Similarly, for $j\in \{0\} \cup [k]$, we have 
\begin{equation*}
\begin{array}{lll}
&\hspace{-5mm}\int_{t_0}^TH_{x_j}(\cdot,\mathrm{X}_{h}^{i-1},u^{i-1},\lambda^{i-1} )\vert_t \Delta^i_{x,h_j}(t)\mathrm{d}t\\
&\hspace{-5mm}=\int_{t_0+h_j}^TH_{x_j}(\cdot,\mathrm{X}_{h}^{i-1},u^{i-1},\lambda^{i-1})\vert_t,\Delta^i_{x,h_0}(t-h_j)\mathrm{d}t\\
&\hspace{-5mm}=\int_{t_0}^{T}H_{x_j}(\cdot,\mathrm{X}_{h}^{i-1},u^{i-1},\lambda^{i-1} )\rvert_{t+h_j}\Delta^i_{x,h_0} (t)\mathrm{d}t,
\end{array}
\end{equation*}
where in the last integral we use the fact that $H_{x_j}\equiv 0$ if evaluated at $t>T$; see \eqref{eq:HamDef}. Recalling \textbf{Step 1} and \eqref{eq:Costate}, we see that $\sum_{j=0}^kH_{x_j}(\cdot,\mathrm{X}_{h}^{i-1},u^{i-1},\lambda^{i-1} )\rvert_{t+h_j}$ is $-\dot{\lambda}^{i-1}(t)$ transposed a.e. Thus, integration of \eqref{eq:Tay2} shows that 
\begin{equation*}
\begin{array}{lll}
\mathcal{J}_1+\mathcal{J}_3&\leq M_7 \|\mu \|_{L^1([t_0,T])}\| \Delta^i_{x,h_0}\|^2_{L^{\infty}([t_0,T])}\\
&-\int_{t_0}^{T}\frac{\mathrm{d}}{\mathrm{d}t}((\lambda^{i-1})^{\top}\Delta^i_{x,h_0})\vert_t\mathrm{d}t\\
% &=M_1(k+1) \|\mu \|_{L^1[t_0,T]}\| \Delta_{x,0}\|_{L^{\infty}[t_0,T]}\\
% &-(\lambda^{i-1})^{\top}\Delta_{x,0}\vert^{T}_{t_0}\\
&=M_7 \|\mu \|_{L^1([t_0,T])}\| \Delta^i_{x,h_0}\|^2_{L^{\infty}([t_0,T])},
\end{array}
\end{equation*}
where the last equality follows because $\Delta_{x,h_0}^i(0)=0$, since $\phi$ is fixed, and $\lambda^{i-1}(T)=0$, by \eqref{eq:Costate}. Overall, we have 
\begin{equation}\label{eq:IntermBound}
\begin{array}{lll}
&\hspace{-8mm} J(u^i)-J(u^{i-1}) \leq -\frac{1}{2}(4\varepsilon_{min}^i+r_{min})\|\Delta^i_u \|_{L^2([t_0,T])}^2\\
&\hspace{16mm}+M_7 \|\mu \|_{L^1([t_0,T])}\| \Delta^i_{x,h_0}\|^2_{L^{\infty}([t_0,T])}.
\end{array}
\end{equation}
We proceed to estimate $\| \Delta^i_{x,h_0}\|^2_{L^{\infty}([t_0,T])}$. In view of \eqref{eq:Dynamics},  
\begin{equation}\label{eq:Deltax}
\begin{array}{lll}
&\hspace{-4mm}\|\Delta^i_{x,h_0}(t) \|\leq \int_{t_0}^t\|f(\cdot,\mathrm{X}_{h}^i,u^i )-f(\cdot,\mathrm{X}_{h}^{i-1},u^{i} ) \|\vert_s\mathrm{d}s\\
&\hspace{-4mm}+\int_{t_0}^t\|f(\cdot,\mathrm{X}_{h}^{i-1},u^i ) - f(\cdot, \mathrm{X}_{h}^{i-1},u^{i-1} ) \|\vert_s\mathrm{d}s =:\mathcal{J}_4+\mathcal{J}_5
\end{array}
\end{equation}
We treat $\mathcal{J}_4$ and $\mathcal{J}_5$ separately. For $\mathcal{J}_5$, by \eqref{eq:CarCond} we obtain
\begin{equation}\label{eq:Deltax1}
\begin{array}{lll}
& \hspace{-6mm} \|f(\cdot,\mathrm{X}_{h}^{i-1},u^i ) - f(\cdot, \mathrm{X}_{h}^{i-1},u^{i-1} ) \|\vert_s\\
& \hspace{-6mm} \leq(\int_0^1 \|f_u(\cdot,\mathrm{X}_{h}^{i-1},ru^i+(1-r)u^{i-1} ) \|\vert_s\mathrm{d}r )\|\Delta_u^i(s)\|\\
& \hspace{-6mm} \leq M_8\mu(s)\|\Delta_u^i(s) \|,
\end{array}
\end{equation}
for some $M_8>0$. Integrating \eqref{eq:Deltax1}, we obtain
\begin{equation}\label{eq:Deltax2}
\begin{array}{lll}
\mathcal{J}_5\leq M_8 \int_{t_0}^T\mu(s)\|\Delta_u^i(s) \| \mathrm{d}s.
\end{array}
\end{equation}
By similar arguments, we obtain 
\begin{equation}\label{eq:Deltax3}
\begin{array}{lll}
\mathcal{J}_4 \leq M_8\int_{t_0}^t\mu(s)(\sum_{j=0}^k \|\Delta^i_{x,h_0}(s-h_j) \|)\mathrm{d}s.
\end{array}
\end{equation}
For $j\in \{ 0\}\cup[k]$, since $\Delta^i_{x,h_0}(t-h_j )=0$ for $t\in[t_0,t_0+h_j]$,
\begin{equation}\label{eq:Deltax4}
\begin{array}{lll}
&\int_{t_0}^t\mu(s)\|\Delta^i_{x,h_0}(s-h_j) \|\mathrm{d}s \\
% &\hspace{5mm}= \int_{t_0}^{t-h_j}\mu(s+h_j)\|\Delta^i_{x,h_0}(s) \|\mathrm{d}s\\
&\hspace{5mm}\leq \int_{t_0}^t\mu(s+h_j )\mathds{1}_{[0,T]}(s+h_j)\|\Delta^i_{x,h_0}(s) \|\mathrm{d}s.
\end{array}
\end{equation}
For $t\in [t_0,T]$, denote $\mathcal{M}(t) = \sum_{j=0}^k\mu(t+h_j)\mathds{1}_{[0,T]}(t+h_j)$. Since $\mathcal{M}\in L^2([t_0,T])$, from \eqref{eq:Deltax2}-\eqref{eq:Deltax4} we have 
\begin{equation*}
\begin{array}{lll}
\|\Delta^i_{x,h_0}(t) \|&\leq M_8 \int_{t_0}^T\mu(s)\|\Delta^i_u(s)\| \mathrm{d}s\\
&+M_8\int_{t_0}^t\mathcal{M}(s)\|\Delta^i_{x,h_0}(s) \| \mathrm{d}s,\ t\in[t_0,T].
\end{array}
\end{equation*}
Applying the Grönwall-Bellman and Hölder inequalities, we conclude that for some $M_9>0$
\begin{equation}\label{eq:Deltax5}
\begin{array}{lll}
\|\Delta^i_{x,h_0}(t) \|^2&\leq  M_9(\int_{t_0}^T\mu(s)\|\Delta_u^i(s)\| \mathrm{d}s )^2\\
&\leq M_9 \|\mu \|_{L^2([t_0,T])}^2\|\Delta^i_u \|^2_{L^2([t_0,T])}.
\end{array}
\end{equation}
From \eqref{eq:IntermBound} and \eqref{eq:Deltax5}, we obtain \eqref{eq:JDecrease}. Choosing $\varepsilon_i^0>0$ such that $\xi_0>0$, since $\{\varepsilon_{min}^i \}_{i\in \mathbb{N}_0}$ is non decreasing, $\{J(u^i) \}_{i=0}^{\infty}$ is decreasing and hence, by Corollary~\ref{cor:CostBounded}, convergent.
\end{proof}

ESSA terminates in a finite (quantifiable) number of steps.
\begin{corollary}\label{Cor:Termin}
Choose $\varepsilon^0_{min}>0$ such that $\xi_0>0$ in \eqref{eq:JDecrease}. Then $\lim_{i\to \infty}\|\Delta^i_u \|_{L^2([t_0,T])}=0$. Moreover, given a lower bound $J_*$ on $J(u)$, $u\in \mathcal{U}$, then the ESSA must terminate in at most $\lfloor( \xi_0\eta_{\text{tol}})^{-1}(J(u^0)- J_* )\rfloor$ steps.
\end{corollary}
\begin{proof}
If it were $\|u^{i_k}-u^{i_k-1} \|\geq \kappa>0$ for some subsequence $\{i_k\}$, then summing \eqref{eq:JDecrease} and noting that $\{\xi_i\}_{i=0}^{\infty}\subseteq \mathbb{R}_{>0}$ is non-decreasing would yield
\begin{equation}\label{eq:Boundsterm}
J(u^n)\leq J(u^0)-\xi_0\sum_{i=1}^n\|\Delta^i_u \|^2_{L^2[t_0,T]},
\end{equation}
which tends to $-\infty$ as $n\to \infty$, leading to a contradiction. Moreover, if the algorithm did not terminate within the first $n'=\lfloor( \xi_0\eta_{\text{tol}})^{-1}(J( u^0)- J_* )\rfloor+1$ steps, substituting $\|\Delta^i_u \|^2_{L^2[t_0,T]}>\eta_{\text{tol}}$, $i\in [n']$, into \eqref{eq:Boundsterm} would yield $J_*\leq J(u^{n'} )<J(u^0)-\xi_0n'\eta_{\text{tol}}\leq J( u^0)-(J( u^0)-J_* )=J_*$, leading again to a contradiction.
\end{proof}

We can offer a stronger bound on the cost evolution for the considered \textbf{Implementation}.
\begin{corollary}
Under the assumptions of Corollary~\ref{Cor:Disc}, if $\varepsilon^0_{min}$ is large enough,
$J(u^i)-J(u^{i-1})\leq -\xi_i\mathfrak{d}_*\|\Delta^i_u\|_{L^\infty([t_0,T])}^2$.
\end{corollary}
% \begin{proof}
% Follows from \eqref{eq:JDecrease} and arguments of Corollary~\ref{Cor:Disc}.
% % since,  as shown in the proof of Corollary~\ref{Cor:Disc}, $\|u^{i}-u^{i-1} \|_{L^{2}([t_0,T] )}^2 \geq \mathfrak{d}_{\mathcal{G}} \|u^{i}-u^{i-1} \|_{L^{\infty}([t_0,T] )}^2 \geq \mathfrak{d}_* \|u^{i}-u^{i-1} \|_{L^{\infty}([t_0,T] )}^2$.
% \end{proof}

We can now prove asymptotic first-order optimality.
\begin{proposition}\label{Prop:AsympFirstOrd}
Assume $C^i\preceq \gamma I$, for some $\gamma>0$, for all the iterations of the ESSA. Then, the sequence $\{u^i \}_{i=0}^{\infty}$ satisfies asymptotically first order optimality conditions, i.e.
\begin{equation*}
\begin{array}{lll}
\lim_{i\to \infty}\| u^i-P_U(u^i-H_u(\cdot,\mathrm{X}^i_{h},u^i,\lambda^{i} ) )\|_{L^2([t_0,T])}=0,
\end{array}
\end{equation*}
for the projection operator $P_U(q) = \operatorname{argmin}_{w\in U}\|q-w \|$.
\end{proposition}
\begin{proof}
The necessary condition for optimality implied by \eqref{eq:PMP}  is that 
$-H_u(\cdot,\mathrm{X}^*_h,u^*,\lambda^*)\vert_t(v-u^*(t)) \leq 0$, for all $v\in U$,
a.e. in $[t_0,T]$. Since $z = P_U(q)$ iff $(q-z)^{\top}(\varpi-z )\leq 0$ for all $\varpi\in U$, this is equivalent to $u^*(t)=P_U(u^*(t)-H_u(\cdot,\mathrm{X}^*_{h},u^*,\lambda^{*} )\vert_t )$. Denote $\zeta^i = P_U(u^i-H_u(\cdot,\mathrm{X}^i_{h},u^i,\lambda^{i} ) ).$
% \begin{equation*}
% \begin{array}{lll}
% \zeta^i = P_U(u^i-H_u(\cdot,\mathrm{X}^i_{h},u^i,\lambda^{i} ) ).
% \end{array}
% \end{equation*}

By Assumption~\ref{Assump:Huu} and \eqref{eq:AugHam}, a.e. in $[t_0,T]$, the function $u\mapsto K(\mathrm{X}^i_{h},u,\lambda^{i-1};u^{i-1},C^i )$ is strictly convex. The minimisation in \textbf{Step 2} of the ESSA implies that, a.e. in $[t_0,T]$,
\begin{equation}\label{eq:uioptae}
\hspace{-2mm} u^{i} =  P_U(u^i-H_u(\cdot,\mathrm{X}^i_{h},u^i,\lambda^{i-1} )-2C^i(u^i-u^{i-1} ) ).
\end{equation}
Therefore, using the fact that $w\mapsto P_U(w)$ is $1$-Lipschitz,
\begin{equation*}
\begin{array}{lll}
&\|u^i-\zeta^i \|_{L^2([t_0,T])}\leq 2\gamma \|\Delta_u^i\|_{L^2([t_0,T])}\\
&\hspace{3mm}+\|H_u(\cdot,\mathrm{X}^i_{h},u^i,\lambda^{i-1} )-H_u(\cdot,\mathrm{X}^i_{h},u^i,\lambda^{i} ) \|_{L^2([t_0,T])}.
\end{array}
\end{equation*}
By Proposition~\ref{Prop:HuCaratheodory},
there exists $M_{10}>0$ such that, for any $t$,
\begin{equation}\label{eq:Hdiff0}
\begin{array}{ll}
&\|H_u(\cdot,\mathrm{X}^i_{h},u^i,\lambda^{i-1} )-H_u(\cdot,\mathrm{X}^i_{h},u^i,\lambda^{i} )\|\vert_t\\
&\hspace{22mm}\leq M_{10}\mu(t)\|\lambda^i(t)-\lambda^{i-1}(t) \|,
\end{array}
\end{equation}
as can be shown by arguments similar to those in the proof of Proposition~\ref{Prop:Regulariz}, whence 
\begin{equation*}
\begin{array}{lll}
&\|H_u(\cdot,\mathrm{X}^i_{h},u^i,\lambda^{i-1} )-H_u(\cdot,\mathrm{X}^i_{h},u^i,\lambda^{i} ) \|_{L^2([t_0,T])} \\
&\hspace{20mm}\leq M_{10}\|\mu \|_{L^2([t_0,T])} \|\lambda^{i}-\lambda^{i-1} \|_{L^{\infty}([t_0,T])}.
\end{array}
\end{equation*}
By arguments akin to those leading to \eqref{eq:Deltax5}, for some $M_{11}>0$, 
\begin{equation}\label{eq:LambFinEst}
\begin{array}{lll}
\|\lambda^{i}-\lambda^{i-1} \|_{L^{\infty}([t_0,T])}\leq M_{11}\|\Delta_u^i \|_{L^2([t_0,T])}.
\end{array}
\end{equation}
In fact, consider $t\in[T-h_1,T]$. By \eqref{eq:Costate}, we have 
\begin{equation}\label{eq:LambdaDiff}
\begin{array}{lll}
&\|\lambda^i(t)-\lambda^{i-1}(t) \|\\
&\leq \int_t^T\|H_{x_0}(\cdot,\mathrm{X}_{h}^{i-1},\lambda^{i-1} ) -H_{x_0}(\cdot,\mathrm{X}_{h}^i,\lambda^i )\|\vert_s\mathrm{d}s.
\end{array}
\end{equation}
Considering the integrand, we have, for some $M_{12}>0$,
\begin{equation}\label{eq:LambdaDiff1}
\begin{array}{lll}
&\hspace{-7mm}\|H_{x_0}(\cdot,\mathrm{X}_{h}^{i-1},\lambda^{i-1} ) -H_{x_0}(\cdot,\mathrm{X}_{h}^i,\lambda^i )\|\vert_s\\
&\hspace{-7mm} \leq M_{12}\mu(s)(\|\mathrm{X}_{h}^i(s)-\mathrm{X}_{h}^{i-1}(s) \| +\|\lambda^i(s)-\lambda^{i-1}(s) \| ).
\end{array}
\end{equation}
In view of \eqref{eq:Deltax5}, it follows that, for some $M_{13}>0$,
\begin{equation*}
\|\mathrm{X}_{h}^i(s)-\mathrm{X}_{h}^{i-1}(s) \| \leq M_{13} \|\mu \|_{L^2([t_0,T])}\|\Delta_u^i \|_{L^2([t_0,T])},
\end{equation*}
whence \eqref{eq:LambdaDiff} and \eqref{eq:LambdaDiff1} yield, for some $M_{14}>0$,
\begin{equation*}
\begin{array}{lll}
&\|\lambda^i(t)-\lambda^{i-1}(t) \| \leq M_{14}\|\Delta_u^i \|_{L^2[t_0,T]}\\
&\hspace{20mm}+M_{14}\int_t^T\mu(s)\|\lambda^{i}(s)-\lambda^{i-1}(s) \|\mathrm{d}s.
\end{array}
\end{equation*}
Employing the Grönwall-Bellman inequality then gives \eqref{eq:LambFinEst}, restricted to the interval $[T-h_1,T]$. Employing the step method backwards in time, as in the proof of Proposition~\ref{Prop:LambdBound}, then yields \eqref{eq:LambFinEst} on $[t_0,T]$.
Now, since $\lim_{t \to \infty}\|\Delta_u^i \|_{L^2([t_0,T])}=0$ by Corollary~\ref{Cor:Termin}, also $\lim_{t \to \infty}\|u^{i}-\zeta^i \|_{L^2([t_0,T])}=0$, which yields the result.
\end{proof}

If the ESSA control sequence converges to $\bar u$, then $\bar u$ can be shown to satisfy first-order optimality. 
\begin{corollary}\label{Cor:PointwiseConv}
Under the conditions of Proposition~\ref{Prop:AsympFirstOrd}, let $\{u^i \}_{i\in \mathbb{N}_0}$ converge to $\bar{u}\in \mathcal{U}$ pointwise a.e. and denote by $\bar{x}$ and $\bar{\lambda}$ the state and co-state corresponding to $\bar{u}$. Then, 
\begin{equation*}
\| \bar{u}-P_U(\bar{u}-H_u(\cdot,\bar{\mathrm{X}}_{h},\bar{u},\bar{\lambda} ) )\|_{L^2[t_0,T]}=0,
\end{equation*}
i.e., $\bar{u}$ satisfies the first-order optimality condition.
\end{corollary}
\begin{proof}
By compactness of $U$ and the dominated convergence theorem, $\|u^i-\bar{u} \|_{L^2([t_0,T])} \to 0$, whence also $\|x^i-\bar{x} \|_{L^{\infty}([t_0,T])} \to 0$ and $\|\lambda^{i}-\bar{\lambda} \|_{L^{\infty}([t_0,T])} \to 0$, in view of \eqref{eq:Deltax5} and \eqref{eq:LambFinEst}. By Proposition~\ref{Prop:AsympFirstOrd}, it suffices to show 
\begin{equation*}
\lim_{i\to \infty}\|H_u(\cdot,\mathrm{X}^i_{h},u^i,\lambda^{i} )-H_u(\cdot,\bar{\mathrm{X}}_{h},\bar{u},\bar{\lambda} ) \|_{L^2[t_0,T]}=0,
\end{equation*}
which follows by arguments akin to \eqref{eq:Hdiff0}--\eqref{eq:LambdaDiff1}. %in the proof of Proposition~\ref{Prop:AsympFirstOrd}.
\end{proof}
% Recalling the \textbf{Implementation}, the assumptions of Corollary~\ref{Cor:PointwiseConv} can be weakened, at the cost of considering a subsequence of $\{u^i \}_{i=0}^{\infty}$.

For the proposed \textbf{Implementation}, convergence of a subsequence to some $\bar u$ is guaranteed.
\begin{corollary}\label{Cor:PointwiseGrid}
Consider the proposed numerical \textbf{Implementation} with an arbitrary grid $\mathcal{G}$. Under the conditions of Proposition~\ref{Prop:AsympFirstOrd}, there exists a subsequence $\{u^{i_k} \}_{k=1}^{\infty}$ which converges pointwise a.e. to a piecewise constant $\bar{u}\in \mathcal{U}$, satisfying the first-order optimality conditions.
\end{corollary}
\begin{proof}
%Denote the nodes of $\mathcal{G}$ by $\{\rho_j \}_{j=0}^{N}$ and 
Denote $\mathcal{I}_j=[\rho_j,\rho_{j+1})$ for $j\in \{0 \}\cup [N-1]$.
For some $\{\vartheta_j^{(i)} \}\subseteq\mathbb{R}$, we have $u^i = \sum_{j=0}^{N-1}\vartheta^{(i)}_j\mathds{1}_{\mathcal{I}_j}$.
% \begin{equation*}
% u^i = \sum_{j=0}^{N-1}\vartheta^{(i)}_j\mathds{1}_{\mathcal{I}_j}.
% \end{equation*}
Since $\{u^i \}_{i=0}^{\infty}$ is bounded in $L^2([t_0,T] )$, it has a weakly convergent subsequence $\{u^{i_k} \}_{k=1}^{\infty}$ satisfying $u^{i_k}\rightharpoonup\bar{u}$. By convexity and compactness of $U$ and the Banach-Saks theorem \cite{megginson2012introduction}, $\bar{u}\in \mathcal{U}$. For $j\in \{0 \}\cup [N-1]$ and $A\subseteq \mathcal{I}_j$ with $\chi(A)>0$, we obtain $\vartheta^{(i_k)}_j = \chi(A)^{-1}\int_{A}u^{i_k}\mathrm{d}t\to\chi(A)^{-1}\int_{A}\bar{u}\mathrm{d}t$.
% \begin{equation*}
% \vartheta^{(i_k)}_j = \chi(A)^{-1}\int_{A}u^{i_k}\mathds{1}_{A}\mathrm{d}t\to\chi(A)^{-1}\int_{A}\bar{u}\mathds{1}_{A}\mathrm{d}t.
% \end{equation*}
Thus, $\{\vartheta^{i_k}_j \}_{k=1}^{\infty}$ converges to $\vartheta_j^*\in U$ and, for all $A\subseteq \mathcal{I}_j$ with $\chi(A)>0$, $\chi(A)^{-1}\int_{A}\bar{u}\mathrm{d}t = \vartheta_j^*$. Thus, $\bar{u} = \sum_{j=0}^{N-1}\vartheta^{*}_j\mathds{1}_{\mathcal{I}_j}$ and $u^{i_k}\to \bar{u}$ pointwise a.e. The result now follows from Corollary~\ref{Cor:PointwiseConv}. 
\end{proof}

Proposition~\ref{Prop:AsympFirstOrd} can be strengthened to guarantee convergence of a subsequence to a control satisfying first-order optimality, regardless of the implementation, if \eqref{eq:Dynamics} is affine in controls, i.e., $f(t,\mathrm{X},u) = f_0(t,\mathrm{X} )+ G(t,\mathrm{X})u$, where $G(t,\mathrm{X})\in \mathbb{R}^{n\times m}$.
%The assumptions of Proposition~\ref{prop:affine} are ubiquitous in OCPs formulated for epidemiological models.
\begin{proposition}\label{prop:affine}
Let \eqref{eq:Dynamics} be affine in controls and consider a cost functional \eqref{eq:PerformanceInd} with $\ell(t,\mathrm{X},u) = \ell_0(t,\mathrm{X} )+u^{\top}Q(t)u$, subject to Assumption~\ref{assum:fandell} on $\ell_0,f_0,G,Q$ and Assumption~\ref{Assump:Huu}. Under the conditions of Proposition~\ref{Prop:AsympFirstOrd}, there exists a subsequence $\{u^{i_k} \}_{k=1}^{\infty}$ that converges weakly to some $\bar{u}\in \mathcal{U}$, satisfying the first-order optimality condition.
\end{proposition}
\begin{proof}
As in the proof of Corollary~\ref{Cor:PointwiseGrid}, let $\{u^{i_k} \}_{k=1}^{\infty}$ satisfy $u^{i_k}\rightharpoonup \bar{u}\in \mathcal{U}$. 
Proposition~\ref{Prop:WPx} shows that $\{x^{i_k}\}_{k=1}^{\infty}$ is uniformly bounded on $[t_0,T]$. Moreover, Assumption~\ref{assum:fandell}  shows that for each $k\in \mathbb{N}$ and $t_0\leq t_1<t_2\leq T$
\begin{equation*}
\begin{array}{lll}
\|x^{i_k}(t_2)-x^{i_k}(t_1) \| &\leq \int_{t_1}^{t_2}\|f(\cdot,\mathrm{X}^{i_k}_h,u^{i_k} ) \|\vert_s \mathrm{d}s\\
&\leq (1+M_1) \int_{t_1}^{t_2}\mu(s)\mathrm{d}s,
\end{array}
\end{equation*}
 whence $\{x^{i_k}\}_{k=1}^{\infty}$ is uniformly equicontinous. Similar arguments hold for $\{\lambda^{i_k}\}_{k=1}^{\infty}$, in view of Proposition~\ref{Prop:LambdBound}. By the Arzelà-Ascoli theorem \cite{megginson2012introduction}, we can extract subsequences of $\{x^{i_k}\}_{k=1}^{\infty}$ and $\{\lambda^{i_k}\}_{k=1}^{\infty}$, which we denote with the same indexing, that converge uniformly to $\bar{x}$ and $\bar{\lambda}$, respectively. Since \eqref{eq:Dynamics} is affine in controls, we have for $t\in [t_0,T]$
 \begin{equation}\label{eq:Differentbarx}
\begin{array}{llll}
&\hspace{-5mm}x^{i_k}(t)-\phi(0) = \int_{t_0}^t(f_0(\cdot,\mathrm{X}_h^{i_k})+G(\cdot,\bar{\mathrm{X}}_h)u^{i_k})\vert_s\mathrm{d}s\\
&\hspace{10mm}+\int_{t_0}^t [(G(\cdot,\mathrm{X}_h^{i_k})-G(\cdot,\bar{\mathrm{X}}_h ))u^{i_k}]\vert_s\mathrm{d}s.
\end{array}
 \end{equation}
The last integral in \eqref{eq:Differentbarx} converges to zero, in view of compactness of $U$ and Assumption~\ref{assum:fandell}. Employing further $u^{i_k}\rightharpoonup \bar{u}$, we have $\bar{x}(t)-\phi(0) =  \int_{t_0}^tf_0(s,\bar{\mathrm{X}}_h(s))+G(s,\bar{\mathrm{X}}_h(s))\bar{u}(s)\mathrm{d}s$, whence $\bar{x}$ is absolutely continuous and solves \eqref{eq:Dynamics} with controller $\bar{u}$. Similar arguments show that $\bar{\lambda}$ is absolutely continuous and solves \eqref{eq:Costate} for $\bar{u}$ and $\bar{x}$.

% We now show that $\bar{u}$ satisfies the first-order optimality condition: for all $v\in U$, $H_u(\cdot,\bar{\mathrm{X}}_h,\bar{u},\bar{\lambda} )\vert_t(v-\bar{u}(t) )\geq 0$. 
Let $A\subseteq[t_0,T]$ with $\chi(A)>0$ and $v\in U$. By \eqref{eq:uioptae}
\begin{equation}\label{eq:Optuik}
\begin{array}{ll}
&2\int_A ( (u^{i_k}-u^{i_k-1})^{\top}C^i(v-u^{i_k}) )\vert_s\mathrm{d}s\\
&\geq \int_A H_u(\cdot,\mathrm{X}_h^{i_k},u^{i_k},\lambda^{i_k-1})\vert_s (u^{i_k}(s)-v) \mathrm{d}s.
\end{array}
\end{equation}
Since $H_u(t,\mathrm{X},u,\lambda)=u^{\top}Q(t)+\lambda^{\top}G(t,x)$,
\begin{equation}\label{eq:Optuik1}
\begin{array}{lll}
& \hspace{-8mm} \int_A H_u(\cdot,\mathrm{X}_h^{i_k},u^{i_k},\lambda^{i_k-1})\vert_sv \mathrm{d}s\\
& \hspace{-8mm} =\int_Av^{\top} (Q^{\top}u^{i_k})\vert_s\mathrm{d}s+\int_Av^{\top} (G^{\top}(\cdot,\mathrm{X}_h^{i_k})\lambda^{i_k-1})\vert_s\mathrm{d}s.
\end{array}
\end{equation}
The first integral on the right-hand side of \eqref{eq:Optuik1} converges to $\int_Av^{\top}(Q^{\top}\bar{u})\vert_s\mathrm{d}s$ since $u^{i_k}\rightharpoonup \bar{u}$. The second integral converges to $\int_Av^{\top}(G^{\top}(\cdot,\bar{\mathrm{X}}_h)\bar{\lambda})\vert_s\mathrm{d}s$, as can be seen by adding and subtracting $(G^{\top}(\cdot,\mathrm{X}_h^{i_k})\lambda^{i_k})\vert_s$ in the integrand and using Assumption~\ref{assum:fandell} for $G$, \eqref{eq:LambFinEst} and the uniform convergence of $\{x^{i_k}\}_{k=1}^{\infty}$ and $\{\lambda^{i_k}\}_{k=1}^{\infty}$. Similar arguments apply to $\int_A H_u(\cdot,\mathrm{X}_h^{i_k},u^{i_k},\lambda^{i_k-1})\vert_su^{i_k}(s) \mathrm{d}s$ with the difference that 
\begin{equation*}
\liminf_{k\to\infty}\int_A ( (u^{i_k})^{\top}Nu^{i_k})\vert_s\mathrm{d}s\geq \int_A (\bar{u}^{\top}Q\bar{u})\vert_s\mathrm{d}s,
\end{equation*}
which follows in view of weak lower semicontinuity, Assumption~\ref{Assump:Huu} and $H_{uu}(t,\mathrm{X},u,\lambda)=Q(t)$. Hence, taking the limit inferior $k\to \infty$ in \eqref{eq:Optuik}, since the left-hand side tends to zero by Corollary~\ref{Cor:Termin}, we have
$0\leq \int_A H_u(\cdot,\bar{\mathrm{X}}_h,\bar{u},\bar{\lambda})\vert_s (v-\bar{u}(s) ) \mathrm{d}s$ for any $A\subseteq[t_0,T]$ with $\chi(A)>0$. Therefore, we have $H_u(\cdot,\bar{\mathrm{X}}_h,\bar{u},\bar{\lambda} )\vert_t(v-\bar{u}(t) )\geq 0$ a.e.
\end{proof}

\begin{figure*}
\centering
\includegraphics[scale=0.22]{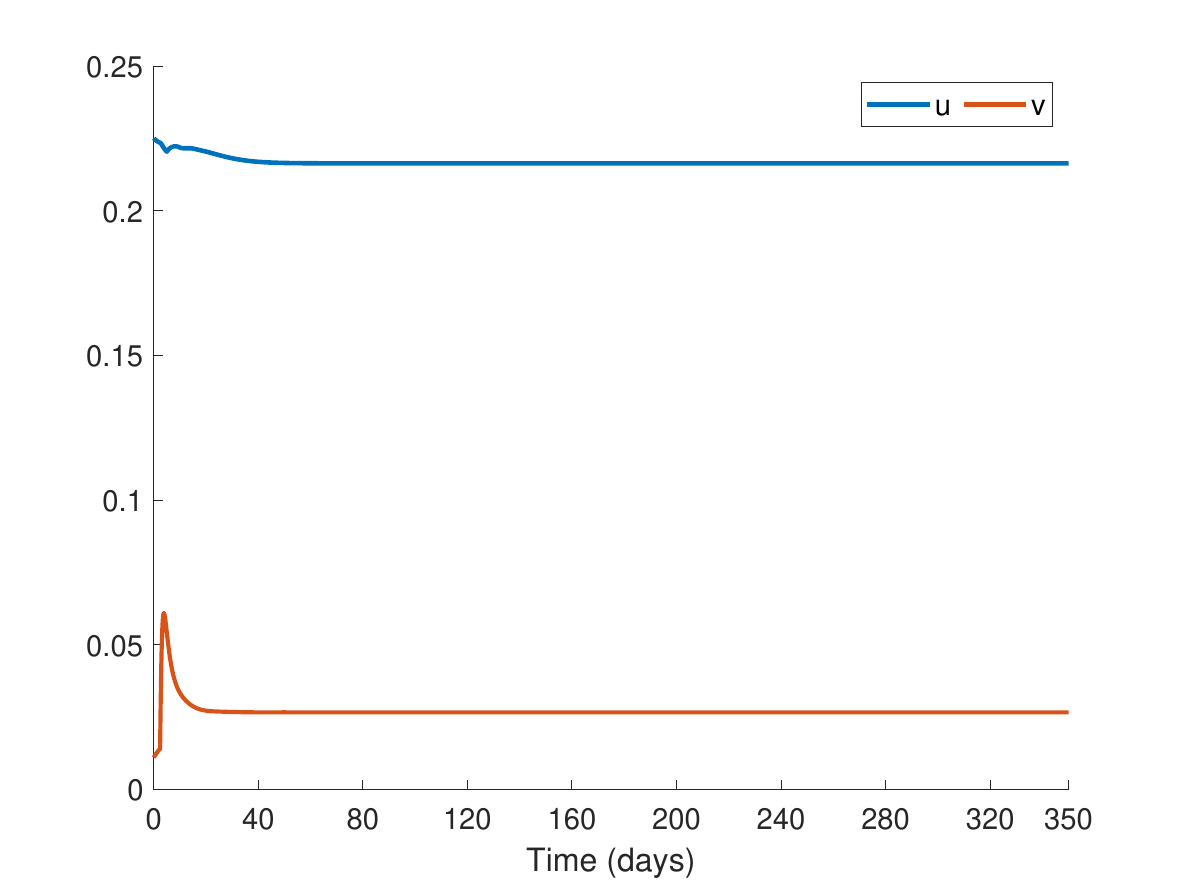}
\includegraphics[scale=0.22]{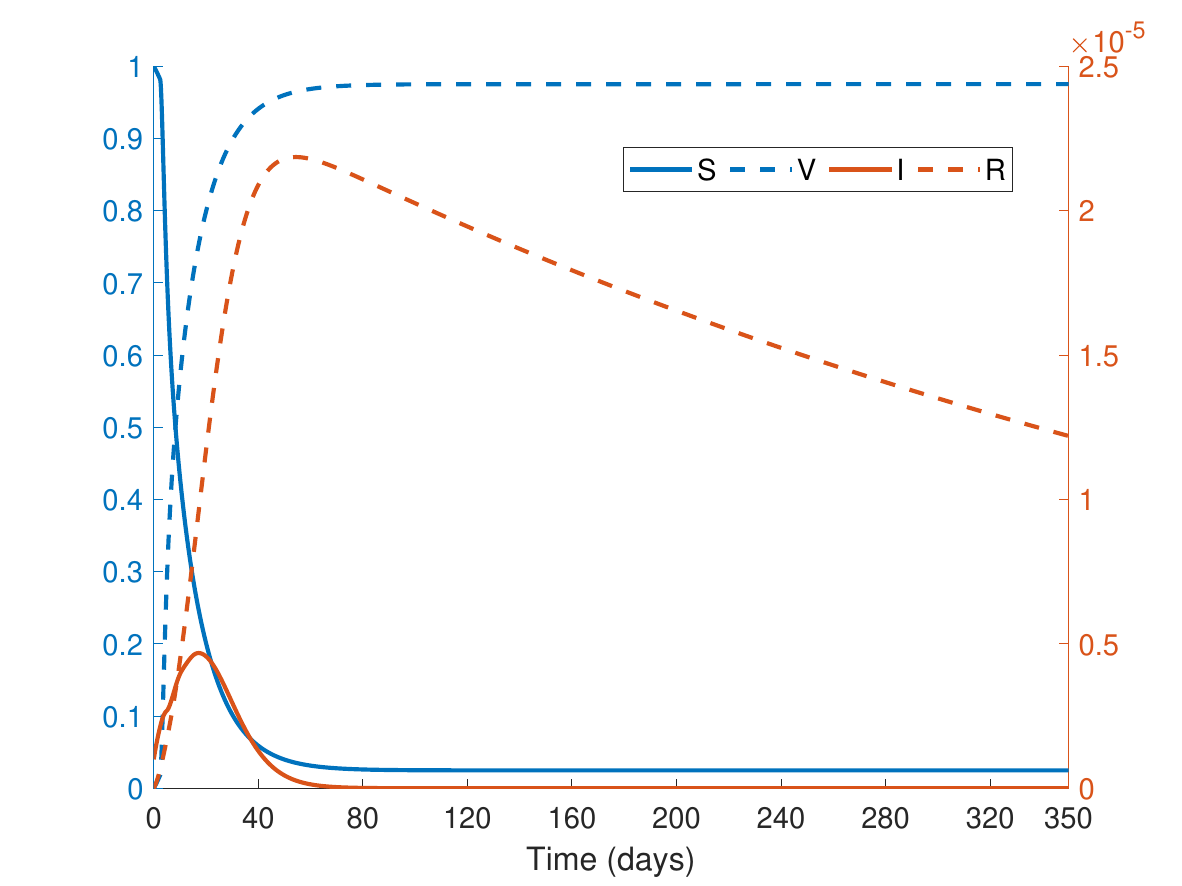} 
\includegraphics[scale=0.22]{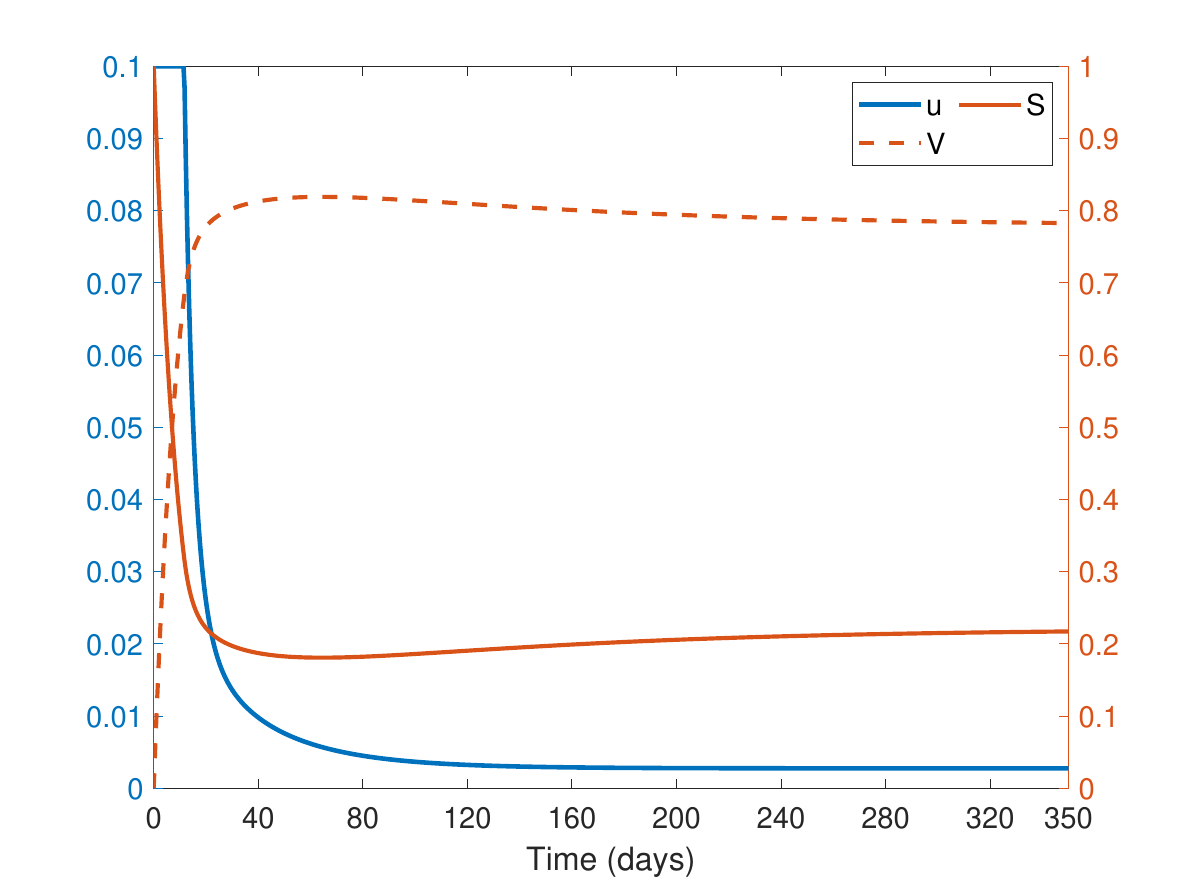}
\includegraphics[scale=0.22]{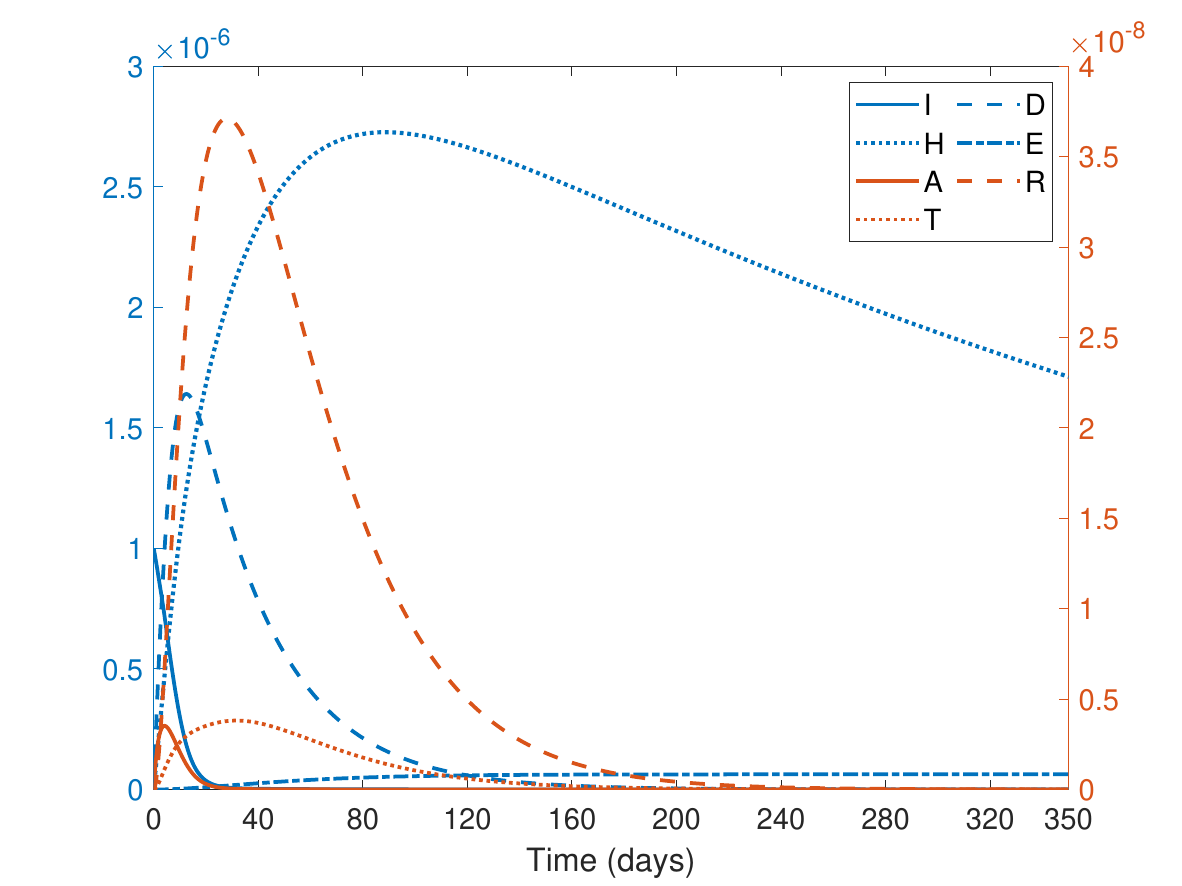}
\caption{Leftmost two panels: Extended SIRV OCP with delays $h_1 = 5$ and $h_2=7$, control bounds $u_{\max}= 0.4$ and $v_{\max}=0.8$, parameters $T =350$ days, $\Lambda= 2.91\cdot 10^{-5}$, $\mu_i = 2.90\cdot10^{-5}$ $\forall i$, $\beta=1$ and $\gamma = 1/6$, $\sigma_V=\sigma_R=1/500$, $\theta_V=0.0013$, $\theta_R=0.0021$, initial conditions $I_0=10^{-6}$, $S_0=1-I_0$, $R_0=0$, $V_0=0$, and cost functional weights $w_I = 10^4$, $w_u = 1$, $w_v=10$. Rightmost two panels: Extended SIDARTHE-V OCP with delays $h_1 = 3$ and $h_2=5$ days, control bound $u_{\max}=0.1$, other parameters and initial conditions as in \cite{CalaCampana2024}, and cost functional weights $w_{y_i} = 10^5$, $y=[I~D~A~R~T]$, and $w_u=1$. Solutions obtained with the ESSA algorithm.}
\label{Fig:sim}
\vspace{-5mm}
\end{figure*}

\section{Optimal Control of Delayed Epidemics}\label{sec:Num_exp}
The ESSA algorithm can efficiently solve OCPs for epidemiological models, where the assumptions of Proposition~\ref{prop:affine} are ubiquitous and hence convergence of a subsequence to a control that satisfies first-order optimality is always guaranteed.
We consider a SIRV model that partitions the population into susceptible, infected, recovered and vaccinated:

\begin{footnotesize}
\begin{equation*}
\begin{cases}
\dot{S} = \Lambda - \beta [1-u] S I_{h_1} + \sigma_R R + \sigma_V V - v S_{h_2} - \mu_S S\\
\dot{I} = \beta [1-u] S I_{h_1} + \theta_V \beta [1-u] V I_{h_1} + \theta_R \beta [1-u] R I_{h_1} - \gamma I - \mu_I I \\
\dot{R} = \gamma I - \sigma_R R - \theta_R \beta [1-u] R I_{h_1} - \mu_R R\\
\dot{V} = v S_{h_2} - \sigma_V V - \theta_V \beta [1-u] V I_{h_1} - \mu_V V
\end{cases}
\end{equation*}
\end{footnotesize}
with birth rate $\Lambda \geq 0$, death rate $\mu_i \geq 0$ for compartment $i$, transmission rate $\beta>0$, recovery rate $\gamma > 0$, waning immunity rates $\sigma_R>0$ and $\sigma_V>0$, infection probabilities $\theta_R \in [0,1]$ and $\theta_V \in [0,1]$ due to imperfect immunity.
The delays capture incubation ($h_1 > 0$) and build-up of vaccine-induced protection ($h_2 > 0$).
We control the stringency $u \in [0, u_{\max}]$ of non-pharmaceutical interventions, such as distancing or mask mandates, and the vaccination rate $v \in [0, v_{\max}]$.
The two left panels in Figure~\ref{Fig:sim} show optimal control and state trajectories for the associated OCP with cost $\frac{1}{2} I(T)^2 + \int_{t_0}^{T} ( w_I I(t) + w_u u(t)^2 + w_v v(t)^2 ) dt$ (see Remark~\ref{rem:terminal}) with $w_I = 10^4$, $w_u = 1$, $w_v=10$ and a suitable parameter choice \cite{CalaCampana2024}.
An optimal combination of non-pharmaceutical interventions and vaccination is selected, where the former are preferred due to the smaller weight in the cost functional, that enables rapid suppression of the outbreak, within $80$ days. Different choices of the cost functional weights lead to different control combinations \cite{Katz2026_extended}.

Introducing an incubation delay $h_1$ and a vaccine build-up delay $h_2$ in the extended SIDARTHE-V model \cite[Example~1]{CalaCampana2024} that includes susceptibles ($S$), different infected classes ($I$, $D$, $A$, $R$, $T$), recovered ($H$), dead ($E$) and vaccinated ($V$), amounts to replacing contagion products $S \mathcal{X}$, $V \mathcal{X}$ and $H \mathcal{X}$ with $S \mathcal{X}_{h_1}$, $V \mathcal{X}_{h_1}$ and $H \mathcal{X}_{h_1}$ respectively, where $\mathcal{X} \in \{I,D,A,R\}$, and the vaccination term $uS$ with $uS_{h_2}$, where the vaccination rate $u \in [0, u_{\max}]$ is the control variable. The cost is $\sum_{i=1}^5 \frac{y_i(T_H)^2}{2} + \int_{t_0}^{T_H} \left( \sum_{i=1}^5 w_{y_i} y_i(t) + w_u u(t)^2 \right) dt$ over the horizon $T_H$, with $y = [I~D~A~R~T]$.
The two right panels in Figure~\ref{Fig:sim} show optimal control and state trajectories for the associated OCP with a suitable parameter choice \cite{CalaCampana2024}. The strong penalisation of infections leads to an aggressive activation of the control, which initially saturates and then remains active, albeit at a low value, thus ensuring that almost $80\%$ of the population has vaccine coverage in spite of waning immunity and that the disease is suppressed within the first $240$ days. A milder control is obtained for smaller weights $w_i$ \cite{Katz2026_extended}.

Thorough numerical simulations illustrating the effect of different time delay values and different cost functional weights on the OCP solution for several epidemiological models are reported as supplementary material online \cite{Katz2026_extended}.

\section{Conclusions}

Our ESSA algorithm extends the S\&S algorithm \cite{Sakawa1980} to tackle OCPs for systems with an arbitrary number of discrete state delays. We have proven theoretical guarantees for ESSA: termination in a finite number of steps, asymptotic first-order optimality and convergence of a subsequence to a control satisfying first-order optimality. We have showcased the application of ESSA to optimally design interventions and vaccination for epidemiological systems affected by time delays capturing incubation and build-up of vaccine-induced protection. Future works includes extensions to the case of distributed delays.

\bibliographystyle{IEEEtran}

\bibliography{diffQMgame}

\end{document}